\documentclass[12pt,reqno]{amsart}
\usepackage{amsfonts,amssymb,amsthm,enumerate,color,comment, pgf,tikz}
\usepackage{amsfonts,amssymb,hyperref,amsthm,enumerate,
color,stmaryrd,pgf,tikz,comment,multirow,mathtools}
\usepackage[left=2.6cm,right=2.6cm,top=3cm,bottom=3cm,bindingoffset=0cm]{geometry}
\newtheorem{thm}{Theorem}[section]
\newtheorem{lem}[thm]{Lemma}

\newtheorem{prop}[thm]{Proposition}
\theoremstyle{remark}
\newtheorem{rem}[thm]{Remark}
\theoremstyle{definition}
\newtheorem{defn}[thm]{Definition}
\newtheorem{exmp}[thm]{Example}
\def\G{\Gamma}
\def\k{\mathbf{k}}
\def\m{\mathbf{m}}
\def\z{\mathbf{z}}
\def\K{\mathbf{K}}
\def\Z{\mathbb{Z}}
\def\N{\mathbb{N}}
\DeclareMathOperator{\aut}{Aut}
\DeclareMathOperator{\cay}{Cay}

\DeclareMathOperator{\parti}{Part}
\newcommand{\sg}[1]{\langle {#1}\rangle}
\begin{document}
\title[Cyclic $m$-DCI-groups and $m$-CI-groups]
{Cyclic $\mathbf{m}$-DCI-groups and $\mathbf{m}$-CI-groups}
\author[I.~Kov\'acs \and L. \v Sinkovec]
{Istv\'an Kov\'acs$^{1}$ \and Luka \v Sinkovec$^{2}$}
\address{I.~Kov\'acs \and L.~\v Sinkovec 
\newline\indent
UP IAM, University of Primorska, Muzejski trg 2, SI-6000 Koper, Slovenia 
\newline\indent
UP FAMNIT, University of Primorska, Glagol\v jaska ulica 8, SI-6000 Koper, Slovenia}
\email{istvan.kovacs@upr.si} 
\email{luka.sinkovec@iam.upr.si}
\thanks{$^1$~This work is supported in part by the Slovenian Research Agency
(research program P1-0285 and research projects N1-0140, J1-2451,
N1-0208, J1-3001, and J1-50000). 
\newline\indent 
$^2$~This work is supported in part by the Slovenian Research Agency (research program P1-0285 and Young Researchers Grant).}
\keywords{Cayley graph, cyclic group, m-CI-group, m-DCI-group}
\subjclass[2020]{05C25, 20B25}

\begin{abstract}
Based on the earlier work of Li (European J. Combin. 1997) and Dobson (Discrete Math. 2008),  in this paper we complete the classification of cyclic $m$-DCI-groups and $m$-CI-groups. For a positive integer $m$ such that $m  \ge 3$, we show that 
the group $\mathbb{Z}_n$ is an $m$-DCI-group if and only if $n$ is not divisible by 
$8$ nor by $p^2$ for any odd prime $p < m$. Furthermore, if $m \ge 6$, then we show that $\mathbb{Z}_n$ is an $m$-CI-group if and only if either $n \in \{ 8, 9, 18 \}$, or 
$n \notin \{ 8, 9, 18 \}$ and $n$ is not divisible by $8$ nor by $p^2$ for any odd prime $p < \frac{m - 1}{2}$.
\end{abstract}
\maketitle
 
\section{Introduction}\label{sec_intro}
Let $G$ be a finite group and $S \subset G$ be a subset not containing 
the identity element of $G$. The {\em Cayley digraph} $\cay(G,S)$ is defined to have 
vertex set $G$ and arcs in the form $(g,sg)$, where $g \in G$ and $s\in S$. 
In the case when $S$ is also inverse-closed, i.e., for every 
$x \in S$, $x^{-1} \in S$, the pair of arcs $(g,sg)$ and $(sg,g)$ can be identified 
with the undirected edge $\{g,sg\}$, in which case $\cay(G,S)$ is regarded as an undirected graph and referred to as a {\em Cayley graph}.  
Every automorphism $\sigma \in \aut(G)$ 
induces an isomorphism between the (di)graphs $\cay(G,S)$ and $\cay(G,S^\alpha)$, in which case $\sigma$ is called a {\em Cayley isomorphism}. 
The converse is not true in general, examples  
of Cayley (di)graphs are known which are isomorphic but there exists no Cayley isomorphism mapping one onto the other.  
A (di)graph $\cay(G,S)$ is called a {\em CI-(di)graph} if for any 
$\cay(G,T)$, whenever $\cay(G,S) \cong \cay(G,T)$, we have 
$S^\sigma=T$ for some $\sigma \in \aut(G)$. 
For a positive integer $m$, the group $G$ has the 
{\em $m$-DCI-property} ({\em $m$-CI-property}, resp.) if all Cayley digraphs 
(Cayley graphs, resp.) on $G$ of valency $m$ are CI-digraphs (CI-graphs, resp.) 
(see \cite[Definition~1.1]{LPX}). By the valency of a Cayley digraph 
$\cay(G,S)$ we mean its in-valency, which is   
equal to the cardinality $|S|$ and coincides with its out-valency. 
Furthermore, $G$ is called an {\em $m$-(D)CI-group} if it has the $i$-(D)CI property for every 
$i \le m$ and a {\em D(CI)-group} when it has the $i$-(D)CI property for every 
positive integer $i$.

The study of finite groups with the $m$-(D)CI property was initiated by Li et al.~\cite{LPX}. 
Both properties attracted considerable attention in the late 90's, 
in particular, it was shown that the finite $m$-DCI-groups belong to 
a relatively restricted list of groups, and a similar list was also derived containing the finite $m$-CI-groups.
For the results from this period, we refer to the survey paper~\cite{L02}.

In this paper, we aim to complete the classification of cyclic $m$-DCI- and $m$-CI-groups. 
We review next the known results.  
In 1967, \'Ad\'am~\cite{Ad} made a conjecture stating that every cyclic group is a 
DCI-group. This conjecture was shown to be false soon after it has been published and 
remarkably, the question as to which cyclic groups are DCI-groups was answered only 30 years later by Muzychuk~\cite{M97}.  

\begin{thm}[\cite{M97}]\label{muzychuk1} \, 
\begin{enumerate}[(1)]
\item The group $\Z_n$ is a DCI-group if and only if $n = k$ or $2k$, where $k$ is square-free.
\item The group $\Z_n$ is a CI-group if and only if either $n \in \{8, 9, 18\}$, or 
$n = k$ or $2k$, where $k$ is square-free.
\end{enumerate}
\end{thm}

A few papers were devoted to cyclic $m$-(D)CI-groups for small $m$ (see~\cite{BT,DFM,FX,L95,S84,S88,T}), in particular,  
Sun~\cite{S84} showed that every cyclic group is a $2$-DCI-group and 
Li~\cite{L95} showed that every cyclic group is a $5$-CI-group.  
At this point we would like to emphasize the difference between 
the $m$-DCI- and the $m$-CI-property for cyclic groups, which lies in 
the fact that by the former property one concerns digraphs, whereas by 
the latter property one concerns only undirected graphs. 
For example, it is known that $\cay(\Z_8,\{1,2,5\})$ is not a CI-digraph~\cite{ET}, showing that $\Z_8$ does not have the $3$-DCI-property. 
On the other hand, as $\Z_n$ is a $5$-CI-group for every $n$, $\Z_8$ has 
the $3$-CI-property (this can be easily checked by simply looking at the possible Cayley 
graphs on $\Z_8$ of valency $3$).

In \cite{L97}, Li determined the positive integers $m$ for which $\Z_{p^2}$ has the $m$-DCI-property. 
The following result was also shown, we present it here as it is stated in~\cite[Theorem~7.2]{L02}.

\begin{thm}[\cite{L02}]\label{L97}
Let $G$ be a cyclic group, and let $G_p$ be the Sylow $p$-subgroup of $G$. Suppose that $G$ has the $m$-DCI-property. If $|G|$ is not a prime-square and $p+1 \le m \le (|G|-1)/2$, then $G_p \cong \Z_p$ or 
$\Z_4$.
\end{thm}

The above theorem provides necessary conditions for a cyclic group to be an $m$-DCI-group, but  
it does not yield the complete classification of cyclic $m$-DCI-groups. 
In his survey, Li wrote (see \cite[p.~319]{L02}): 
\begin{quote}
`A natural question is, for a positive integer $m$, 
which cyclic groups are $m$-DCI-groups, and which cyclic groups are 
$m$-CI-groups?'    
\end{quote} 
Motivated by Theorem~\ref{L97}, he also proposed  
a conjecture, which contains sufficient conditions for a cyclic group to be an $m$-DCI- and $m$-CI-group, respectively (see~\cite[Conjecture~7.3]{L02}). This conjecture was proved later 
by Dobson~\cite{D} and it reads as follows.

\begin{thm}[\cite{D}]\label{D}
Let $n=n_1n_2 \in \N$ such that $\gcd(n_1, n_2) = 1$ and $n_1$ divides $4k$, where $k$ is odd and square-free. 
\begin{enumerate}[(1)]
\item If $p > m$ for every prime divisor $p$ of $n_2$, then $\Z_n$ is an 
$m$-DCI-group. 
\item If $p > 2m$ for every prime divisor $p$ of $n_2$, then $\Z_n$ is a 
$2m$-CI-group.
\end{enumerate}
\end{thm}

However, it turns out that in order to classify the 
cyclic $m$-DCI- and $m$-CI-groups, the sufficient conditions in Theorem~\ref{D} need to be strengthened. Our main result is the following generalization.

\begin{thm}\label{main1}
Let $n=n_1n_2 \in \mathbb{N}$ such that $\gcd(n_1, n_2) = 1$ and $n_1$ divides $4k$, where $k$ is odd and square-free. 
\begin{enumerate}[(1)]
\item If $p \ge m$ for every prime divisor $p$ of $n_2$, then $\Z_n$ is an 
$m$-DCI-group. 
\item If $2p+1 \ge m$ for every prime divisor $p$ of $n_2$, then $\Z_n$ is an  
$m$-CI-group.
\end{enumerate}
\end{thm}

Combining these stronger sufficient conditions with the known necessary conditions, we  
complete the classification of cyclic $m$-DCI- and $m$-CI-groups.

\begin{thm}\label{main2}\, 
\begin{enumerate}[(1)]
\item If $m \in \N$ and $m \ge 3$, then the group $\Z_n$ is an $m$-DCI-group 
if and only if $n$ is not divisible by $8$ nor by $p^2$ for any odd prime $p < m$.
\item If $m \in \N$ and $m \ge 6$, then the group $\Z_n$ is an $m$-CI-group 
if and only if either $n \in \{ 8, 9, 18 \}$, or $n \notin \{ 8, 9, 18 \}$ and it is not 
divisible by $8$ nor by $p^2$ for any odd prime $p < \frac{m - 1}{2}$.
\end{enumerate}
\end{thm}

\begin{rem}\label{remark1}
The necessity part of Theorem~\ref{main2} follows directly from known constructions of non-CI-digraphs and graphs 
discovered in~\cite{A,ET,LPX}.  All these examples will appear in 
the course of the proof of Theorem~\ref{main2}. 
\end{rem}

\begin{rem}\label{remark2}
Comparing Theorem~\ref{main2} with the classification of cyclic DCI- and CI-groups given 
in~Theorem~\ref{muzychuk1}, we obtain that the group $\Z_n$ is an $m$-DCI-group for 
$m \ge 3$ but not a DCI-group if there exist a prime $p$ such that $p \ge m$ and 
$p^2 \mid n$; and also that $\Z_n$ is an $m$-CI-group for $m \ge 6$ but not a CI-group 
if $n \notin \{8, 9, 18\}$, and there exists a prime $p$ such that 
$p \ge (m-1)/2$ and $p^2 \mid n$.
\end{rem}

Dobson~\cite{D} obtained Theorem~\ref{D} using permutation group theoretical techniques. 
Our proof of Theorem~\ref{main1} is independent, it relies on a necessary 
and sufficient criterion for two circulant digraphs to be isomorphic due to Muzychuk~\cite{M04}. In the next section we give a short overview of Muzychuk's criterion, Theorems~\ref{main1} and \ref{main2} will then be derived in 
Section~\ref{sec_proofs}. 
\section{Isomorphic circulant digraphs}\label{sec_preliminaries}

\subsection{Key spaces}\label{subsec_keys}
For $n \in \N$, we use the symbol $[n]$ to denote the set $\{1,\ldots,n\}$.
Let $n=p_1^{t_1}\cdots p_l^{t_l}$ be the prime decomposition of $n$.
In what follows, $\Z_n$ will be identified with the direct sum 
$\Z_{p_1^{t_1}} \oplus \cdots \oplus \Z_{p_l^{t_l}}$.  Therefore, an element 
$x  \in \Z_n$ will also be represented as the $l$-tuple $x=(x_1,\ldots,x_l)$, where  
$x_i \in \Z_{p_i^{t_i}}$ for every $i \in [l]$. 

\begin{defn}\label{Kn}
The \emph{key space} $\K_n$ is the direct product 
\[
\K_n:= \K_{p_1^{t_1}} \times \cdots \times \K_{p_l^{t_l}},
\]
where for a prime power $p^t$, $t \ge 1$, 
$\K_{p^t}$ consists of the $t$-tuples $(k_{1}, \ldots, k_{t}) \in \N_{0}^{t}$ such that,
\begin{itemize}
\item $\forall i \in [t],~0 \le k_{i} < i$,
\item $\forall i \in [t-1],~k_{i} \le k_{i+1}$.
\end{itemize}
\end{defn}

The elements of $\K_n$ are called \emph{keys}. 
Thus a key $\k \in \K_n$ is a two-dimensional integer array, which will be written as  
\[
\k=(\k_1,\ldots,\k_l)=(k_{ij})_{i \in [l], j \in [t_i]},
\] 
where $\k_i=(k_{i1},\ldots,k_{it_i}) \in \K_{p_i^{t_i}}$. 

The \emph{zero key} $\z$ in $\K_n$ is defined as $\z=(z_{ij})_{i\in [l],j \in [t_i]}$  
with $z_{ij}=0$ for every $i \in [l], j \in [t_i]$. If $n \equiv 4 \pmod 8$, then the 
\emph{almost zero key} $\z_*$ in 
$\K_n$ is defined as $\z_*=(z_{ij})_{i\in [l], j \in [t_i]}$ with  
\[
z_{ij}=\begin{cases}
1 & \text{if}~p_i^{t_i}=4~\text{and}~j=2, \\
0 & \text{otherwise}.
\end{cases}
\]

For two keys $\k, \m \in \K_n$, we write $\k \le \m$ if 
$k_{ij} \le m_{ij}$, for every  $i \in [l]$ and 
$j \in [t_i]$. The relation $\le$ is a partial order on $\K_n$. 
The poset $(\K_n,\le)$ is a lattice with meet $\wedge$ and join $\vee$ 
operations defined as
\[
\forall i \in [l],\, \forall j \in [t_i],~
(\k \wedge \m)_{ij}=\min(k_{ij},m_{ij})~\text{and}~(\k \vee \m)_{ij}=
\max(k_{ij},m_{ij}).
\]

\subsection{Key partitions}\label{subsec_keypartitions}
The set of all partitions of $\Z_n$ will be denoted by $\parti(\Z_n)$. 
For two partitions $\Sigma, \Delta \in \parti(\Z_n)$,   
$\Delta$ is said to be a \emph{refinement} of $\Sigma$, denoted by 
$\Sigma \sqsubseteq \Delta$, if every class of $\Sigma$ is union of classes of $\Delta$. The relation $\sqsubseteq$ is a partial order on $\parti(\Z_n)$. 
The poset $(\parti(\Z_n),\sqsubseteq)$ is a lattice. We write 
$\Sigma \wedge \Delta$ for the meet and  $\Sigma \vee \Delta$ for the join 
of the partitions $\Sigma$ and $\Delta$.

\begin{defn}\label{Sigmak-p}
Let $n=p^t$ for a prime $p$ and let $\k=(k_1,\ldots,k_t) \in \K_{p^t}$. 
The \emph{key partition} $\Sigma(\k)$ is the partition of 
$\Z_{p^t}$ defined as
\[
\Sigma(\k)=\big\{ \{0\} \big\} \cup \big\{ P_{k_{\alpha(x)}}+x \, \mid \, x \in \Z_{p^t} 
\setminus \{0\} \big\},
\]
where for $x \in \Z_{p^t} \setminus \{0\}$, $\alpha(x)$ is the number in $[t]$ defined by the 
property that $p^{\alpha(x)}$ is the order of $x$ as an element in the group $(\Z_{p^t},+)$;   and for $j \in \{0,1,\ldots,t\}$, $P_j$ is the subgroup of $\Z_{p^t}$ of order $p^j$.
\end{defn}

\begin{defn}\label{Sigmak-n}
Let $n\in \N$ be a number greater than $1$ with prime decomposition $n=p_1^{t_1}\cdots p_l^{t_l}$ and 
let $\k=(\k_1,\ldots,\k_l) \in \K_n$. Then the \emph{key partition} $\Sigma(\k)$ is the partition of $\Z_n$ defined as  
\[
\Sigma(\k) = 
\big\{\, (S_1, \ldots , S_l) \, \mid \, S_i \in \Sigma(\k_i),  i \in [l]\, \big\}.
\]
\end{defn}

\begin{prop}[{\rm \cite[Proposition~2.2]{M04}}]\label{prop}
Let $n \in \N$ and $\k, \m \in \K_n$. If $\k \le \m$ then $\Sigma(\m) \sqsubseteq \Sigma(\k)$.
\end{prop}

\begin{defn}
Let $n \in \N$ and let $\Pi \in \text{Part}(\Z_n)$ be an arbitrary partition of $\Z_n$. 
The \emph{key partition of} $\Pi$ is the coarsest key partition $\Sigma(\m)$ that  
refines $\Pi$; or more formally,  
\begin{itemize}
\item $\Pi \sqsubseteq \Sigma(\m)$. 
\item If $\Pi \sqsubseteq \Sigma(\k)$ for some $\k \in \K_{n}$, then 
$\Sigma(\m) \sqsubseteq \Sigma(\k)$.
\end{itemize}
The above key $\m$ is referred to as the \emph{key of $\Pi$}.
\end{defn}

\begin{defn}\label{kS}
Let $n \in \N$ and $S$ be a nonempty subset of $\Z_n$. 
The \emph{key} of $S$, denoted by $\k(S)$, is the key of the partition 
$\{ S, \Z_{n} \setminus S \}$. 
\end{defn}

\subsection{Generalized multipliers}\label{subsec_genmulti}

We consider first the generalized multipliers of cyclic groups of prime power order. 

\begin{defn}\label{vecm-p}
Let $p \in \N$ be a prime and $t \in \N$. A \emph{generalized multiplier} of 
$\Z_{p^t}$ is a $t$-tuple $\vec{m} = (m_1, \ldots, m_t) \in \N^t$ such that 
$\gcd(m_i, p) = 1$ for every $i \in [t]$. 
\end{defn}

The set of all generalized multipliers of $\Z_{p^t}$ is denoted by $\Z_{p^t}^{**}$. 
We associate next a permutation of $\Z_{p^t}$ with any 
generalized multiplier in $\Z_{p^t}^{**}$.

\begin{defn}\label{fm-p}
Let $p \in \N$ be a prime and $t \in \N$. For $\vec{m} = (m_1, \ldots, m_t) \in \Z_{p^t}^{**}$, the \emph{generalized multiplier function} $f_{\vec{m}}$ is the permutation of 
$\Z_{p^t}$, defined as
\[
\forall x \in \Z_{p^t},~
x^{f_{\Vec{m}}} = \sum_{i=0}^{t - 1} m_{t - i} x_{i} p^{i} \!\!\!\!\pmod{p^t}, 
\]
where $\sum_{i=0}^{t - 1} x_{i} p^{i}$ is the $p$-adic decomposition of $x$, i.e., $x_i \in \{ 0, \ldots, p - 1 \}$ for every $i \in \{ 0,\ldots, t - 1 \}$.
\end{defn}

It follows from the above definition that, if $\vec{m}=(m_1,\ldots,m_t)$ and 
$\vec{m'}=(m'_1,\ldots,m'_t)$ are two generalized 
multipliers of $\Z_{p^t}$ such that 
\[
\forall i \in [t],~m_i \equiv m'_i \!\!\!\! \pmod{p^i},
\]
then $f_{\vec{m}}=f_{\vec{m'}}$. Thus, to have all generalized multiplier functions, it is sufficient to 
consider only those generalized multipliers $\vec{m}=(m_1,\ldots,m_t)$, which satisfy  
$m_i \in [p^i-1]$ for every $i \in [t]$. 

Fix a key $\k \in \K_{p^t}$. We define a subset $\Z_{p^t}^{**}(\k) \subseteq \Z_{p^t}^{**}$ as 
follows. If $t>1$, then let
\[
\Z_{p^t}^{**}(\k) =
\big\{ (m_1,\ldots,m_t)\in \Z_{p^t}^{**} \, \mid \, 
\forall i \in [t-1],~m_{i+1} \equiv m_i\!\!\!\!\pmod {p^{i-k_{i+1}}} \, \big\};
\]
and if $t=1$, then let $\Z_p^{**}(\k)=\Z_p^{**}$.

A $\Sigma(\k)$-class different from $\{0\}$ is in the form $P_{k_i}+x$, where $P_{k_i}$ is the subgroup of $\Z_{p^t}$ of order $p^{k_i}$, and $x$ is a non-zero element of order $p^i$ in $(\Z_{p^t},+)$. It is easy to see that for 
every $\vec{m}=(m_1,\ldots,m_t) \in \Z_{p^t}^{**}(\k)$, $x^{f_{\vec{m}}}$ has order $p^i$, and 
\[
(P_{k_i}+x)^{f_{\vec{m}}}=P_{k_i}+x^{f_{\vec{m}}}.
\]
These show that $f_{\vec{m}}$ induces a permutation of 
$\Sigma(\k)$. Furthermore, one can show that, if $\vec{m'}=(m'_1,\ldots,m'_t)$ is another generalized multiplier from $\Z_{p^t}^{**}(\k)$ such that 
$m'_i \equiv m_i\!\! \pmod {p^{i-k_i}}$, then 
$(P_{k_i}+x)^{f_{\vec{m}}}=(P_{k_i}+x)^{f_{\vec{m'}}}$. 
In view of this, to have all permutations of $\Sigma(\k)$ induced by the generalized multipliers in $\Z_{p^t}^{**}(\k)$, it suffices to 
consider only those generalized multipliers $\vec{m}=(m_1,\ldots,m_t) \in \Z_{p^t}^{**}(\k)$,  which satisfy $m_i \in [p^{i-k_i}-1]$ for every $i \in [t]$. 
We refer to the latter generalized multipliers as \emph{genuine} and denote their set by 
$\Z_{p^t}^{**}(\k)^\circ$. More formally, 
\[  
\Z_{p^t}^{**}(\k)^\circ = 
\big\{ (m_1,\ldots,m_t)\in \Z_{p^t}^{**}(\k) \, \mid \, 
\forall i \in [t],~m_i \in [p^{i-k_i}-1]\, \big\}. 
\]

Let $n \in \N$ be an arbitrary number with prime decomposition $n=p_1^{t_1}\cdots p_l^{t_l}$, and let 
$\k=(\k_1,\ldots,\k_l) \in \K_n$ be an arbitrary key. The above concepts are generalized to $\Z_n$ as follows:
\begin{align*}
\Z_{n}^{**}  &= \Z_{p_1^{t_1}}^{**} \times \cdots \times \Z_{p_l^{t_l}}^{**},  \\
\Z_{n}^{**}(\k) &= \Z_{p_1^{t_1}}^{**}(\k_1) \times \cdots \times 
\Z_{p_l^{t_l}}^{**}(\k_l),  \\
\Z_{n}^{**}(\k)^\circ &= \Z_{p_1^{t_1}}^{**}(\k_1)^\circ \times \cdots \times 
\Z_{p_l^{t_l}}^{**}(\k_l)^\circ.
\end{align*}

For $\vec{m} = (\Vec{m_1}, \ldots, \Vec{m_t}) \in \Z_{n}^{**}$, the generalized 
multiplier function $f_{\vec{m}}$ is the permutation of $\Z_n$ acting as 
\[
\forall (x_1,\ldots,x_l) \in \Z_n,~
(x_1,\ldots,x_l)^{f_{\Vec{m}}} = 
\big(\, x_{1}^{f_{\Vec{m_{1}}}}, \ldots, x_{l}^{f_{\Vec{m_{l}}}} \big).
\] 

\begin{defn}
Let $n \in \N$ and let $\k \in \K_n$. The \emph{solving set} $P(\k)$ is the 
following set of permutations of $\Z_n$:
\[
P(\k) = \{ f_{\Vec{m}} \, \mid \, \Vec{m} \in \Z_{n}^{**}(\k)^\circ \}.
\]
\end{defn}

The example below is meant to illustrate all the concepts introduced in this section.

\begin{exmp}\label{Kp2}
We determine below the solving sets $P(\k)$, where $\k \in \K_{p^2}$ and $p$ is a 
fixed prime. 

Due to Definition~\ref{Kn}, $\K_{p^2}$ consists of two keys: $(0,0)$ and $(0,1)$. 
Let $\Vec{m}=(m_1,m_2) \in \Z_{p^2}^{**}((0,i))^\circ$, where $i \in \{0,1\}$. 
Then 
\[
m_1 \in [p-1],~m_2 \in [p^{2-i}-1],~\gcd(m_2,p)=1~\text{and}~m_2 \equiv m_1\!\!\!\!\pmod {p^{1-i}}.
\]
Due to Definition~\ref{fm-p}, 
$(x_0+x_1p)^{f_{\Vec{m}}}=m_2x_0+m_1x_1p$. Note that, if $i=0$, then 
$m_1$ is uniquely determined by $m_2$. Moreover, 
$m_1p \equiv m_2p \pmod {p^2}$, and hence we have $(x_0+x_1p)^{f_{\Vec{m}}}=
m_2(x_0+x_1p)$. We conclude that 
\begin{align}
P((0,0)) &= \big\{ x \mapsto mx \, \mid \, 
m \in [p^2-1], \gcd(m,p)=1 \big\}=\aut(\Z_{p^2}), \nonumber \\ 
P((0,1)) &= \big\{ x_0+x_1p \mapsto m_2x_0+m_1 x_1p \, \mid \, 
m_1,m_2 \in [p-1]\big\}.
\label{eq:01}
\end{align} 
\hfill $\square$
\end{exmp}

In the next lemma we generalize the above example to solving sets 
$P(\z)$, where $\z$ is the zero key in $\K_{p^t}$, i.e., $\z=(0,\ldots,0)$.

\begin{lem}\label{zerokey}
Let $p$ be a prime and $t \in \N$. 
For the zero key $\z \in \K_{p^t}$, $P(\z) = \aut(\Z_{p^t})$.
\end{lem}
\begin{proof}
Let $f_{\Vec{m}} \in P(\z)$, where 
$\Vec{m}=(m_1,\ldots,m_t) \in \Z_{p^t}^{**}(\z)^\circ$. 
Note that, for every $i \in [t], m_i \in [p^i-1]$ and $\gcd(m_i,p)=1$. We are going to show 
that $x^{f_{\Vec{m}}} \equiv m_tx \pmod {p^t}$ holds for every $x \in \Z_{p^t}$. 

Fix an arbitrary $x \in \Z_{p^t}$ and let 
$x=\sum_{i=0}^{t-1}x_ip^i$ be its $p$-adic decomposition. 
Due to Definition~\ref{fm-p}, 
\begin{equation}\label{eq:1}
x^{f_{\Vec{m}}}=m_tx_0+m_{t-1}x_1p+\cdots+m_1x_{t-1}p^{t-1}.
\end{equation}

On the other hand, it follows from the definition of $\Z_{p^t}^{**}(\z)^\circ$ that,  
\[
\forall i \in [t-1],~m_{i+1} \equiv m_i \!\!\!\!\pmod {p^i}.
\]
This yields  
\[
\forall i \in [t],~m_t \equiv m_i \!\!\!\!\pmod {p^i}.
\]
Thus, $m_tp^{t-i} \equiv m_ip^{t-i} \pmod {p^t}$ for every 
$i \in [t]$, and then substituting this in \eqref{eq:1}, we obtain that 
$x^{f_{\Vec{m}}} \equiv m_tx \pmod {p^t}$, as required.
\end{proof}

\subsection{Muzychuk's criterion}\label{subsec_criterion}
Now we are prepared to present Muzychuk's criterion for two 
circulant digraphs to be isomorphic.

\begin{thm}[{\rm \cite{M04}}] \label{muzychuk2}
Let $n \in \N$ such that $n > 1$ and let $\cay(\Z_n,S)$ and $\cay(\Z_n,T)$ be two circulant 
digraphs. Then we have the following.
\begin{enumerate}[(1)]
\item If $\k(S) \neq \k(T)$, then $\cay(\Z_n,S) \ncong \cay(\Z_n,T)$.
\item If $\k(S) = \k(T) = \k$, then the following statements are equivalent.
\begin{enumerate}[(i)]
\item $\cay(\Z_n,S) \cong \cay(\Z_n,T)$.
\item There exists a permutation $f_{\Vec{m}} \in P(\k)$ such that
\[
\cay(\Z_n,S)^{f_{\Vec{m}}}=\cay(\Z_n,T).
\]
\item There exists a permutation $f_{\Vec{m}} \in P(\k)$ such that 
$S^{f_{\Vec{m}}} = T$.
\end{enumerate}
\end{enumerate}
\end{thm}

Using the above theorem, we obtain the following sufficient condition 
for a circulant digraph to be CI.

\begin{lem}\label{sufficient}
Let $n \in \N$ such that $n > 1$ and $S \subseteq \Z_n \setminus \{0\}$ be a subset such that  
$\k(S)=\z$ or $\z_*$ (i.e., the zero key or the almost zero key, 
resp. in $\Z_n$). Then $\cay(\Z_n,S)$ is a CI-digraph.
\end{lem}

\begin{proof}
Suppose that $\cay(\Z_n,S) \cong \cay(\Z_n,T)$ for some subset $T \subset \Z_n$.  
We have to show that $S^\sigma=T$ for some $\sigma \in \aut(\Z_n)$. 

By Theorem~\ref{muzychuk2}, $\k(T) = \k(S)$ and there exists 
$f_{\Vec{m}} \in P(\k)$, such that $S^{f_{\Vec{m}}} = T$. 
Let $n$ have prime decomposition $n=p_1^{t_1}\cdots p_l^{t_l}$. 
Then $\k=(\k_1,\ldots,\k_l) \in \K_{p_1^{t_1}} \times \cdots \times \K_{p_l^{t_l}}$, and 
the action of $f_{\Vec{m}}$ on $\Z_n$ can be described as 
\[
(x_1,\ldots,x_l)^{f_{\Vec{m}}}=\big(\, 
x_1^{f_{\Vec{m_1}}},\ldots,x_l^{f_{\Vec{m_l}}}\, \big),
\]
where for every $i \in [l]$, $f_{\Vec{m_i}} \in P(\k_i)$.
If $\k_i$ is the zero key in $\K_{p_i^{t_i}}$, then 
$f_{\Vec{m_i}} \in \aut(\Z_{p_i^{t_i}})$ by Lemma~\ref{zerokey}. 
On the other hand, as $\k=\z$ or $\k=\z_*$, we conclude that 
$\k_i$ is the zero key, unless, $p_i^{t_i}=4$ and $\k_i=(0,1)$. 
In the later case, we see in \eqref{eq:01} that $f_{\Vec{m_i}}$ is the 
identity permutation of $\Z_4$. We conclude that 
$f_{\Vec{m}} \in \aut(\Z_n)$. 
\end{proof}
\section{Proof of the main results}\label{sec_proofs}

\begin{lem}[{\cite[Lemma~2.1]{LPX}}]\label{LPX}
Let $G$ be a finite group, and let $S, T \subseteq G \setminus \{1_G\}$. 
Then $\cay(G,S) \cong \cay(G,T)$ if and only if 
$\cay(\sg{S},S) \cong \cay(\sg{T},T)$. 
\end{lem}

\begin{lem}\label{reduction}
For every subset $S   \subseteq \Z_n \setminus \{0\}$,  
$\cay(\Z_n,S)$ is a CI-digraph if and only if $\cay(\sg{S},S)$ is a CI-digraph.
\end{lem}

\begin{proof}
($\Longrightarrow$) Suppose that $\cay(\Z_n, S)$ is a CI-digraph and 
$\cay(\sg{S},S) \cong \cay(\sg{S},T)$ for some subset $T \subset \sg{S}$. 
Since $\cay(\sg{S},S)$ is connected, it follows that $\cay(\sg{S},T)$ is also connected 
and thus $\sg{T}=\sg{S}$. By Lemma~\ref{LPX}, 
$\cay(\Z_n, S) \cong \cay(\Z_n, T)$. As $\cay(\Z_n, S)$ is a CI-digraph, there exists a Cayley isomorphism $\sigma \in \aut(\Z_n)$ between $\cay(\Z_n, S)$ and 
$\cay(\Z_n,T)$. The automorphism $\sigma$ maps $\sg{S}$ to itself, hence 
its restriction to $\sg{S}$ induces a Cayley isomorphism between 
$\cay(\sg{S},S)$ and $\cay(\sg{S},T)$, and so 
$\cay(\sg{S},S)$ is indeed a CI-digraph.

($\Longleftarrow$) Suppose that $\cay(\sg{S},S)$ is a CI-digraph and 
$\cay(\Z_n,S) \cong \cay(\Z_n,T)$ for some subset $T \subset \Z_n$. 
By Lemma~\ref{LPX}, $\cay(\sg{S},S) \cong \cay(\sg{T},T)$. Since $\Z_n$ is a cyclic group, $\sg{S}=\sg{T}$. 
As $\cay(\sg{S},S)$ is a CI-digraph, there exists a Cayley isomorphism 
$\sigma \in \aut(\sg{S})$ between $\cay(\sg{S},S)$ and $\cay(\sg{S},T)$.
It is not hard to show that $\sigma$ can be extended to an automorphism of 
$\Z_n$, or in other words, there exists an automorphism $\sigma' \in \aut(\Z_n)$ such 
that its restriction to $\sg{S}$ is equal to $\sigma$. We obtain that $\sigma'$ is a Cayley isomorphism between $\cay(\Z_n, S)$ and $\cay(\Z_n, T)$, so 
$\cay(\Z_n, S)$ is a CI-digraph.
\end{proof}

The following lemma follows from \cite[Lemma~25]{D}.

\begin{lem} \label{lexi1}
Let $H \le \Z_n$ be a subgroup and let $s \in \Z_n$ be any element outside $H$. 
Then $\cay(\Z_n, H + s)$ is a CI-digraph and $\cay(\Z_n, (H + s) \cup (H - s))$ is a CI-graph. 
\end{lem}

\begin{lem} \label{lexi2}
Let $n \in \N$ be an even number and let $p$ be an odd prime such that 
$p^2 \mid n$. Suppose that $S = (P + s) \cup (P - s) \cup \{ \frac{n}{2} \}$, 
where $P \le \Z_n$ is the subgroup of order $p$, $P+s \ne P-s$ and 
$\sg{S}=\Z_n$. Then $\cay(\Z_n, S)$ is a CI-graph.
\end{lem}

\begin{proof}
For the sake of simplicity we set $\G=\cay(\Z_n,S)$. 
If $n=18$, then $\G$ is a CI-graph by Theorem~\ref{muzychuk1}, 
hence in the rest of the proof we assume that $n > 18$.
Call an edge $\{ x, y \}$ of $\G$ a {\em half edge} if $y - x = \frac{n}{2}$. 
\medskip

\noindent{\bf Claim.}~
{\it An edge of $\G$ is contained in a subgraph  
isomorphic to $K_{p,p}$ if and only if it is not a half edge.}
\medskip

Let $e$ be a half edge of $\G$, i.e., $e=(x,x+\frac{n}{2})$ for some $x \in \Z_n$. 
Assume for the moment that $e$ is contained in a subgraph isomorphic to $K_{p,p}$. 
Then the latter subgraph contains a $4$-cycle in the form 
$(x,x+\frac{n}{2},y_1,y_2)$ with no half edge besides $e$. This means that there are  
elements $z, z_1, z_2 \in P$ satisfying one of the following equalities:
\[
x+z  \pm s=y_2=y_1+z_2 \pm s=x+\frac{n}{2}+z_1 \pm s+z_2 \pm s.
\]
This shows that $s+\frac{n}{2} \in P$ or $3s+\frac{n}{2} \in P$, and hence 
$2s \in P$ or $6s \in P$. If $2s \in P$, then the order of $s$ is a divisor of $2p$, and as 
$\Z_n=\sg{S}$, we get $n=2p$. This contradicts that $p^2 \mid n$. 
If $6s \in P$, then the order of $s$ is a divisor of $6p$. Using also that $p^2 \mid n$, 
we find that $p=3$ and $n=18$. This, however, is excluded. 

Suppose second that $e=\{x,y\}$ is not a half edge, i.e., $y=x+z+\varepsilon s$, where 
$z \in P$ and $\varepsilon \in \{1,-1\}$. It is straightforward to check that 
the subgraph of $\G$ induced by the set $(P+x) \cup (P+y)$ contains 
$e$ and it is isomorphic to $K_{p,p}$. This completes the proof of the claim. 
\medskip
 
The above argument also shows that the biparts of the subgraphs of $\G$ isomorphic 
to $K_{p,p}$ partition the vertex set $\Z_n$ into the $P$-cosets, which therefore, 
form an $\aut(\G)$-invariant partition. 

Now suppose that $\cay(\Z_n,S) \cong \cay(\Z_n,T)$ for some subset 
$T \subset \Z_n$. It follows in turn that, the partition of $\Z_n$ into the $P$-cosets is also $\aut(\cay(\Z_n,T))$-invariant, $T \setminus \{\frac{n}{2}\}=
(P+t) \cup (P-t)$ for a suitable element $t \in \Z_n$, and 
\[
\cay(\Z_n,(P+s) \cup (P-s)) \cong \cay(\Z_n,(P+t) \cup (P-t)).
\] 
By Lemma~\ref{lexi1}, $T \setminus \{ \frac{n}{2} \} = (S \setminus \{ \frac{n}{2} \})^{\sigma}$ for some automorphism $\sigma$ of $\Z_n$, and hence 
$T=S^\sigma$. 
\end{proof}

We are ready to derive our main theorems. 
Given a positive integer $k$, in the proofs below we denote by $\pi(k)$ the 
set of all prime divisors of $k$. For a prime $p$, the {\em $p$-part} of $k$ is 
the largest power of $p$ dividing $k$. 
 
\begin{proof}[Proof of Theorem~\ref{main1}]
If $n_2 = 1$, then Theorem~\ref{main1} follows from Theorem~\ref{muzychuk1}. Thus for the rest of the proof we assume that $n_2 > 1$. Let $S \subset \Z_n$ such that 
$0 \notin S$ and $|S| \le m$. By Lemma~\ref{reduction}, it suffices to show that 
$\cay(\sg{S},S)$ is a CI-(di)graph. We derive this in three steps. 
Write $\sg{S}=\Z_{n'}$.
\medskip

\noindent{\bf Step~1.}~
{\it $\cay(\Z_{n'},S)$ is a CI-(di)graph, unless $n_2$ is odd.}
\medskip

Assume that $n_2$ is even. Then $|S| \le 2$ if $\cay(\Z_{n'}, S)$ is a digraph and $|S| \le 5$ if it is a graph. Thus the claim in Step~1 follows from the known facts that 
every cyclic group is $2$-DCI-group~\cite{S84} and also a 
$5$-CI-group~\cite{L95}.
\medskip

For the rest of the proof we assume that $n_2$ is odd. Let $p$ be the smallest 
prime divisor of $n_2$ and $P$ be the subgroup of $\Z_n$ of order $p$. 
\medskip

\noindent{\bf Step~2.}~  
{\it $\cay(\Z_{n'},S)$ is a CI-(di)graph, unless $p^2 \mid n'$
and one of the following holds: 
\begin{enumerate}
\item[(i)] $S = P + s$, for some $s \in \Z_{n'}$, 
\item[(ii)] $S = (P + s) \cup (P - s)$, for some $s \in \Z_{n'}$ such that 
$P+s \ne P-s$, 
\item[(iii)] $n'$ is even and $S = (P + s) \cup (P - s) \cup \{ \frac{n'}{2} \}$, for some $s \in \Z_{n'}$ such that $P+s \ne P-s$.
\end{enumerate}}
\medskip

Assume that $S$ is none of the sets described in the cases (i)--(iii) above. 
We show below that $\cay(\Z_{n'},S)$ is a CI-(di)graph.

Let $n'$ have prime decomposition $n' = p_{1}^{t_1} \cdots p_{l}^{t_l}$. 
Let $\k = (\k_1, \ldots, \k_l) = \k(S)$ be the key of the set $S$. 
Suppose that $\k_i \ne \z$, the zero key in $\K_{p_i^{t_i}}$. Then $t_i > 1$. 
Assume for the moment that $p_i > 2$. Since $p_i^2 \, \mid \, n'$, it follows that $p_i \in \pi(n_2)$. As $\sg{S}=\Z_{n'}$, there is an element $s \in S$ such that the order of $s$ in the group 
$\Z_{n'}$ has $p_i$-part $p_i^{t_i}$. Let 
$\m=(\m_1,\ldots,\m_l)$ be the key of $\Z_{n'}$ defined as for every $j \in [l]$,
\[
\m_j=\begin{cases} (0,\ldots,0) & \text{if}~j \ne i, \\
(0,\ldots,0,1) & \text{if}~j=i.
\end{cases}
\]
The cell of $\Sigma(\m)$ containing $s$ is equal to $P_i + s$, where $P_i \le \Z_{n'}$ and $|P_i| = p_i$. Also, $\m \leq \k$, hence by Proposition~\ref{prop}, $\Sigma(\k) \sqsubseteq \Sigma(\m)$. Let $C$ be the cell of $\Sigma(\k)$ containing $s$. It follows from the definition of $\k=\k(S)$ that 
\[
P_i + s \subseteq C \subseteq S \implies p_i \leq |S|.
\] 

Suppose first that $\cay(\Z_{n'}, S)$ is a digraph. Since $p_i \in \pi(n_2)$, it follows that 
$|S| \leq p_i$ also holds, and we conclude that $S = P_i + s$. If $p_j \in \pi(n_2)$ and $p_j \ne p_i$, then $p_i=|S| \leq p_j$. Thus $p_i = p$, the smallest prime divisor of $n_2$, and so $S$ is a set described in case (i), which is impossible. 

Suppose now that $\cay(\Z_{n'}, S)$ is a graph. Then $S = -S$, and so 
\[
(P_i + s) \cup (P_i - s) \subseteq C \subseteq S \implies 2p_i \leq |S|.
\] 
For the last implication observe that $P_i + s \ne P_i - s$, for otherwise, $2s \in P_i$,
contradicting that the order of $s$ has $p_i$-part $p_i^{t_i} > p_i$. As $p_i \in \pi(n_2)$, $|S| \leq 2p_i + 1$ also holds, so we conclude that $S = (P_i + s) \cup (P_i - s)$ or $n'$ is even and $S = (P_i + s) \cup (P_i - s) \cup \{ \frac{n'}{2} \}$. If $p_j \in \pi(n_2)$ and $p_j \ne p_i$, then $2p_i \leq |S| \leq 2p_j + 1$. It follows that $p_i=p$ and $S$ is a set in the form given in either case (ii) or case (iii), which is again impossible. We obtain that $p_i = 2$. But as $n_2$ is odd, $2$ must be a divisor of $n_1$, and so $t_i = 2$.

We conclude that $\k = \z$, the zero key of $\Z_{n'}$, or the $2$-part of $n'$ is $4$ 
and $\k=\z_*$, 
the almost zero key in $\K_{n'}$. By Lemma~\ref{sufficient}, $\cay(\Z_{n'}, S)$ is a 
CI-(di)graph.
\medskip

\noindent{\bf Step~3.}~ 
{\it $\cay(\Z_{n'}, S)$ is a CI-(di)graph.}
\medskip

In view of Step~2, we may assume that $p^2 \, | \, n'$. 
If $\cay(\Z_{n'}, S)$ is a digraph, then $S$ is a set described in case (i) of Step~2, and the statement follows from Lemma~\ref{lexi1}. Suppose that $\cay(\Z_{n'}, S)$ is a graph. If $S$ is a set described in case (ii) of Step~2, then the statement follows again from Lemma~\ref{lexi1}; while if it is a set described in case (iii), then the statement 
follows from Lemma~\ref{lexi2}.
\end{proof}

\begin{proof}[Proof of case (1) in Theorem~\ref{main2}]
($\Longrightarrow$) This is equivalent to the following implication: If $8 \mid n$ or $p^2 \mid n$ for some odd prime $p < m$, then $\Z_n$ is not an $m$-DCI-group. 

Suppose first that $8 \mid n$. Then there is a subset $S \subset \Z_n$ such that $\langle S \rangle \cong \Z_8$ and 
\[
\cay(\langle S \rangle, S) \cong \cay(\Z_8, \{ 1, 2, 5 \}),
\]
the non-CI-digraph discussed in~\cite{ET}. By Lemma~\ref{reduction}, $\cay(\Z_n,S)$ is a non-CI-digraph, and so $\Z_n$ is not an $m$-DCI-group (recall that $m \ge 3$). 

Now suppose that $p^2 \mid n$ for some odd prime $p < m$. 
Then there is a subset $S \subset \Z_n$ such that $\sg{S}\cong \Z_{p^2}$ and 
\[
\cay(\sg{S}, S) \cong \cay(\Z_{p^2},(\sg{p}+1) \cup \{p\}),
\]
the non-CI-digraph discussed in \cite{LPX} (see the proof of \cite[Theorem~1.4]{LPX}).
By Lemma~\ref{reduction}, $\cay(\Z_n,S)$ is a non-CI-digraph, and so 
$\Z_n$ is not an $m$-DCI-group. 

($\Longleftarrow$) Define $n_2$ to be $1$ if there is no odd prime $p \in \pi(n)$ such that $p^2$ divides $n$, and let $n_2 = \prod_{p \in \pi(n) \atop 2 < p < n_p} n_p$ otherwise, where $n_p$ denotes the $p$-part of $n$. 
Let $n_1 = \frac{n}{n_2}$. It follows that $\gcd(n_1, n_2)=1$ and $n_1 \mid 4k$, where $k$ is odd and square-free. Note that, for every $p \in \pi(n_2)$, $p^2 \mid n$, 
implying that $p \ge m$. Thus Theorem~\ref{main1} can be applied to $\Z_n$ and this gives that $\Z_n$ is an $m$-DCI-group. 
\end{proof}

\begin{proof}[Proof of case (2) in Theorem~\ref{main2}]
($\Longrightarrow$) This is equivalent to the following implication: If $n \notin \{ 8, 9, 18 \}$ and $8 \mid n$ or $p^2 \mid n$ for some odd prime $p < \frac{m-1}{2}$, then $\Z_n$ is not an m-CI-group. 

If $8 \mid n$, then the graph 
\[
\cay\big(\Z_n,\big\{1,n-1,2,n-2,\frac{n}{2}-1,\frac{n}{2}+1\big\}\big)
\]
is a non-CI-graph as shown in \cite[Example~1.22]{A}. 
Thus $\Z_n$ is not an $m$-CI-group (recall that $m \ge 6$).

Now suppose that $p^2 \mid n$ for some odd prime $p < \frac{m-1}{2}$. 
Let $p=3$ (so $m \ge 8$). The graph 
\[
\cay\big(\Z_n,\big\{1,n-1,3,n-3,\frac{n}{3}+1,\frac{n}{3}-1,
\frac{2n}{3}+1,\frac{2n}{3}-1,\big\}\big)
\]
is a non-CI-graph as shown in \cite[Example~1.23]{A}, and hence  
$\Z_n$ is not an $m$-CI-group.

Let $p \ge 5$ (so $m \ge 2p+2$). In this case, 
there is a subset $S \subset \Z_n$ such that $\sg{S}\cong \Z_{p^2}$ and 
\[
\cay(\sg{S}, S) \cong \cay\big(\Z_{p^2}, (\sg{p}+1) \cup (\sg{p} - 1) \cup \{p,-p\}\big),
\]
the non-CI-graph discussed in \cite{LPX} (see the proof of \cite[Theorem~1.5]{LPX}).
By Lemma~\ref{reduction}, $\cay(\Z_n,S)$ is a non-CI-graph, and so 
$\Z_n$ is not an $m$-CI-group. 

($\Longleftarrow$) If $n \in \{ 8, 9, 18 \}$, then $\Z_n$ is a CI-group, so the statement holds. Assume that $n \notin \{ 8, 9, 18 \}$. Define $n_2$ to be $1$ if there is no odd prime $p \in \pi(n)$ such that $p^2$ divides $n$, and let $n_2 = \prod_{p \in \pi(n) \atop 2 < p < n_p}n_p$ otherwise, where $n_p$ denotes the $p$-part of $n$. 
Let $n_1 = \frac{n}{n_2}$. It follows that $\gcd(n_1,n_2)=1$ and $n_1 \mid 4k$, where $k$ is odd and square-free. Note that, for every $p \in \pi(n_2)$, $p^2 \mid n$, 
implying that $2p+1 \ge m$. Thus Theorem~\ref{main1} can be applied to $\Z_n$ and this gives that $\Z_n$ is an $m$-CI-group.
\end{proof}

\section*{Acknowledgement}

The authors are grateful to the anonymous reviewers for their valuable comments and suggestions which helped to improve the presentation.


\begin{thebibliography}{999}
%
\bibitem{Ad}
A. \'Ad\'am, 
Research problem 2--10, 
\emph{J. Combin Theory} \textbf{1} (1967), 309.
%
\bibitem{A}
B. Alspach,
Isomorphism and Cayley graphs on abelian groups, 
in: G. Hahn and G. Sabidussi (eds.), 
\emph{Graph Symmetry: Algebraic Methods and Applications,
NATO ASI Series C} \textbf{497}, Springer, Dordrecht, 1997, 1–-22.
%
\bibitem{BT}
F. Boesch, R. Tindell, 
Circulants and their connectivities, 
\emph{J. Graph Theory} \textbf{8} (1984), 487--499.
%
\bibitem{DFM}
C. Delorme, O. Favaron, M. Maheo,
Isomorphisms of Cayley multigraphs of degree $4$ on finite abelian groups,
\emph{European J. Combin.} \textbf{13} (1992), 59--61.
%
\bibitem{D}
E. Dobson,
On isomorphisms of circulant digraphs of bounded degree,
\emph{Discrete Math.} \textbf{308} (2008), 6047--6055.
%
\bibitem{ET}
B. Elspas, J. Turner,
Graphs with circulant adjacency matrices,
\emph{J. Combin. Theory} \textbf{9} (1970), 297–-307.
%
\bibitem{FX}
X. G. Fang, M. Y. Xu,
On isomorphisms of Cayley graphs of small valency,
\emph{Algebra Colloq.} \textbf{1} (1994), 67--76.
%
\bibitem{L95}]
C. H. Li,
Isomorphisms and classification of Cayley graphs of small valencies on finite abelian groups,
\emph{Australas. J. Combin.} \textbf{12} (1995), 3--14.
%
\bibitem{L97}
C. H. Li,
The cyclic groups with the $m$-DCI-property,
\emph{European J. Combin.} \textbf{18} (1997), 655--665.
%
\bibitem{L02}
C. H. Li,
On isomorphisms of finite Cayley graphs -- a survey,
\emph{Discrete Math.} \textbf{256} (2002), 301--334.
%
%
\bibitem{LPX}
C. H. Li, C. E. Praeger, M. Y. Xu,
On finite groups with the Cayley isomorphism property,
\emph{J. Graph Theory} \textbf{27} (1998), 21–-31.
%
%
\bibitem{M97}
M. Muzychuk,
On \'Ad\'am's conjecture for circulant graphs,
\emph{Discrete Math.} \textbf{167/168} (1997), 497--510; corrigendum
\textbf{176} (1997), 285--298.
%
\bibitem{M04}
M.  Muzychuk,
A solution of the isomorphism problem for circulant graphs,
\emph{Proc. London Math. Soc.} \textbf{88} (2004), 1--41.
%
\bibitem{S84}
L. Sun,
Isomorphisms of circulants with degree $2$,
\emph{J. Beijing Ins. Technol.} \textbf{9} (1984), 42--46.
%
\bibitem{S88}
L. Sun,
Isomorphisms of circulants graphs,
\emph{Chinese Ann. Math.}~\textbf{9A} (1988), 259--266.
%
\bibitem{T}
S. Toida,
A note on \'Ad\'am's conjecture,
\emph{J. Combin. Theory Ser. B} \textbf{23} (1977), 239--246.
\end{thebibliography}
\end{document}